\documentclass[11pt,reqno]{article}

\usepackage{amssymb}
\usepackage{amsmath}
\usepackage{amsthm}
\usepackage{mathrsfs}
\usepackage[all]{xy}
\usepackage{pstricks}

\newtheorem{theorem}{Theorem}[section]
\newtheorem{lemma}[theorem]{Lemma}
\newtheorem{corollary}[theorem]{Corollary}
\newtheorem{proposition}[theorem]{Proposition}

\newtheorem{problem}[theorem]{Problem}

\theoremstyle{remark}

\newtheorem{example}[theorem]{Example}
\newtheorem{remark}[theorem]{Remark}
\newtheorem*{ack}{Acknowledgements}

\numberwithin{equation}{section}

\def\vv{\mathbf{v}}
\def\ww{\mathbf{w}}
\def\xx{\mathbf{x}}
\def\es{\varnothing}

\def\st{^\ast}

\def\ol#1{\overline{#1}}
\def\wh#1{\widehat{#1}}
\def\Fr{Fra\"{\i}ss\'e\ }
\def\C{\mathscr{C}}

\renewcommand\leq{\leqslant}
\renewcommand\geq{\geqslant}

\DeclareMathOperator\End{End} \DeclareMathOperator\Aut{Aut} \DeclareMathOperator\Sym{Sym} 
\DeclareMathOperator\Self{Self}   
\DeclareMathOperator\Con{Con} \DeclareMathOperator\Fl{Flim}  

\title{A universality result for endomorphism monoids of\\ some ultrahomogeneous structures}

\author{Igor Dolinka \lowercase{and} Dragan Ma\v sulovi\'c\\[1mm]
{\footnotesize\it Department of Mathematics and Informatics, University of Novi Sad,}\\[-1mm]
{\footnotesize\it Trg Dositeja Obradovi\'ca 4, 21000 Novi Sad, Serbia}\\[-1mm]
{\footnotesize{\it e-mail: }dockie@dmi.uns.ac.rs; masul@dmi.uns.ac.rs}}

\date{}

\begin{document}

\maketitle

\begin{abstract}
We devise a fairly general sufficient condition ensuring that the endomorphism monoid of a countably infinite
ultrahomogeneous structure (i.e.\ a \Fr limit) embeds all countable semigroups. This approach provides us not only with
a framework unifying the previous scattered results in this vein, but actually yields new applications for endomorphism
monoids of the (rational) Urysohn space and the countable universal ultrahomogeneous semilattice.

\medskip

\noindent{\it Keywords:} endomorphism monoid; \Fr limit; pushout

\medskip

\noindent 2010 {\it Mathematics subject classification:} Primary 20M20; Secondary 03C15, 08A35, 18A30, 20M50
\end{abstract}

\section*{Introduction}

The principal source of motivation for this paper stems from fascinating objects of model theory called \emph{\Fr
limits}. Namely, by a well-known result of R.\ \Fr \cite{F1,F2} (see also \cite{H}), if $\C$ is a countable set of
finitely generated first-order structures of a fixed countable signature which is (up to isomorphism) closed under
taking finitely generated substructures, enjoys the \emph{amalgamation property} (AP) and the \emph{joint embedding
property} (JEP), then there is a unique countable structure $F$ with the following two properties:
\begin{itemize}
\item the set of all finitely generated substructures of $F$ coincides (up to isomorphism) with $\C$,

\item $F$ is \emph{ultrahomogeneous}, meaning that any isomorphism between its finitely generated substructures extends
to an automorphism of $F$.
\end{itemize}
Such $F$ is called the \emph{$\C$-universal ultrahomogeneous structure} or, following \cite{H}, the \emph{\Fr limit} of
$\C$. The corresponding class $\C$ with the above properties is called a \emph{\Fr class}. Conversely, it is known that
every countably infinite ultrahomogeneous structure arises in this way: it is simply the limit of the class of
structures isomorphic to its finitely generated substructures. Of course, if all of this happens within a locally
finite class of structures (e.g.\ relational structures with finitely many constants), then it suffices to replace
`finitely generated' by `finite' in the above definitions.

Historically, the two oldest examples of \Fr limits are the \emph{rational Urysohn space} $\mathbb{U}_\mathbb{Q}$, the
limit of the class of all finite metric spaces with rational distances, and the linear order of the rationals,
$\mathbb{Q}$, which is the limit of the class of all finite linear orders. Other well-known classes admitting \Fr
limits include finite simple graphs, finite posets, and finite semilattices, resulting in the random graph $R$
\cite{C1,C2}, the random (generic) poset $\mathbb{P}$ \cite{DK,Sch}, and the countable universal homogeneous
semilattice $\Omega$ \cite{DKT}, respectively.

More recently, searching for a broader perspective on the notion of ultrahomogeneity, P.\ J.\ Cameron and J.\ Ne\v
set\v ril \cite{CN} introduced the concept of \emph{homomorphism-homogeneity}. A structure $A$ is
homomorphism-homogeneous if any homomorphism $B\to A$ defined on a finitely generated substructure $B$ of $A$ can be
extended to an endomorphism of $A$. In particular, any endomorphism of $B$ extends to an endomorphism of $A$. Now, a
somewhat more sophisticated question arises whether one can select these extensions in a `regular' fashion, in the
sense that it is possible to choose, for each $\varphi\in\End(B)$, an extension $\wh\varphi\in\End(A)$ such that the
mapping $\varphi\mapsto\wh\varphi$ is a monoid embedding $\End(B)\to\End(A)$. The same question applies to an arbitrary
substructure $B$ of $A$.

Here we consider the case when $A$ is a countably infinite ultrahomogeneous structure, that is, a \Fr limit. Since $A$
is universal for a certain class of countable structures, it is to be expected that $\End(A)$, the endomorphism monoid
of $A$, will exhibit a very rich structure and to some extent `inherit' the universal properties of $A$. In fact, if
for some (countably) infinite substructure $B$ of $A$ there is an embedding $\End(B)\to\End(A)$ and $B$ admits any
self-map of the set $B$ as an endomorphism, then $\Self(\aleph_0)$, the monoid of all self-maps of a countably infinite
set, embeds into $\End(A)$. Since any countable semigroup embeds into $\Self(\aleph_0)$ (by the semigroup analogue of
the Cayley theorem) it would then follow that the monoid $\End(A)$ is \emph{countably universal}, i.e.\ that it
contains a copy of each countable semigroup.

For example, it was shown in \cite{BDD} that this is true in the case of the random graph $R$, while the same
conclusion is reached in \cite{D} for the generic poset $\mathbb{P}$. Furthermore, Uspenskij \cite{Us} proved that
every topological group with a countable base is isomorphic to a topological subgroup of the isometry (i.e.\
automorphism) group of the \emph{Urysohn space} $\mathbb{U}$ \cite{Ka,Ur}, which implies that every countable group
embeds into $\Aut(\mathbb{U}_\mathbb{Q})$. Also, for any countably infinite ultrahomogeneous simple graph $G$ (a
complete list of these is given by Lachlan and Woodrow \cite{LW}) it is true that the symmetric group $\Sym(\aleph_0)$,
and thus any countable group, embeds into $\Aut(G)$, see \cite{He}. This is, however, no longer true if we move into
the realm of monoids and semigroups: it follows from \cite{Mu} that $R$ is actually the \emph{only} countable
ultrahomogeneous graph whose endomorphism monoid embeds all countable semigroups. Moreover, an easy exercise shows that
no nontrivial finite group embeds into $\End(\mathbb{Q})$, the monoid of all order-preserving mappings of $\mathbb{Q}$.

In this paper we develop a system of conditions on a \Fr class $\C$ under which its limit $F=\Fl(\C)$ has the property
that $\End(F)$ embeds the endomorphism monoid of any of its substructures, with universality consequences as sketched
above. After introducing the required preliminary notions and constructions, in Section 2 we state our general method.
The concrete applications are presented in Section 3. First, we review the results of \cite{BDD,D} in the light of the
presented general approach. We also derive some new consequences by proving that both $\End(\mathbb{U}_\mathbb{Q})$ and
$\End(\Omega)$ contain copies of $\Self(\aleph_0)$. Some open problems are discussed as well.

\section{Preliminaries: pushouts, amalgams, amalgamated sums}

For a \Fr class $\C$, let $\ol\C$ denote the class of all countable structures all of whose finitely generated
substructures belong to $\C$. We shall be concerned with \Fr classes $\C$ for which $\ol\C$, considered as a category
of first-order structures, admits certain constructions and properties (see \cite{Mac} for the basics of category
theory). The most fundamental such construction will be that of a \emph{pushout}. Let $X,Y,Z$ be objects of a category
$\mathbf{C}$, along with morphisms $f:X\to Y$ and $g:X\to Z$; such a configuration
$$
\xymatrix{Y & X \ar[l]_{f} \ar[r]^{g} & Z}
$$
is referred to as a \emph{span}. The pushout of this span consists of an object $P$ and two morphisms $i_1:Y\to P$ and
$i_2:Z\to P$ such that:
\begin{itemize}
\item[(a)] the diagram
$$
\xymatrix{Y \ar[r]^{i_1} & P\\
X \ar[u]^f \ar[r]_g & Z \ar[u]_{i_2}}
$$
commutes,
\item[(b)] for any object $Q$ and morphisms $j_1:Y\to Q$ and $j_2:Z\to Q$ for which the part of following diagram
involving $X,Y,Z,Q$ is commutative, there exists a unique morphism $u:P\to Q$ making the whole diagram
$$
\xymatrix{&& Q\\
Y \ar[r]^{i_1} \ar@/^2ex/[urr]^{j_1} & P \ar[ru]^u & \\
X \ar[u]^f \ar[r]_g & Z \ar[u]_{i_2} \ar@/_2ex/[ruu]_{j_2} & }
$$
commutative.
\end{itemize}
A standard argument shows that the pushout is unique up to isomorphism. The diagram from item (a) above will be called
a \emph{pushout square} (in $\mathbf{C}$). Throughout the paper the composition of morphisms will follow the
right-to-left convention, so that if $\alpha:A\to B$ and $\beta:B\to C$, then their composition is $\beta\alpha:A\to
C$. Thus the condition (b) asserts the existence of a unique morphism $u$ such that $ui_1=j_1$ and $ui_2=j_2$, where
$j_1f=j_2g$.

In the remaining part of the paper we are going to work with \emph{concrete} categories, whose objects are first-order
structures (of a particular similarity type) and morphisms are homomorphisms of such structures. It is crucial to
stress at this point that in the definition of a pushout, the homomorphism $u$ need not to be an embedding (or
injective) even if all the other morphisms occurring in the previous diagram (that is, $f,g,i_1,i_2,j_1,j_2$) are
embeddings.

Recall that an \emph{amalgam} (in a concrete category) is a span $(A,B,C,f,g)$ such that $f:B\to A$ and $g:B\to C$ are
embeddings. If $A,B,C\in\C$ for some class $\C$, then we have an amalgam \emph{in} $\C$. The \emph{amalgamation
property} for $\C$, mentioned earlier, asserts that any amalgam in $\C$ can be embedded into a structure $D\in\C$,
i.e.\ that there are embeddings $j_1:A\to D$ and $j_2:C\to D$ such that $j_1f=j_2g$. If $\C$ is a class of finitely
generated structures with the AP (for example, a \Fr class), then it is a part of folklore in model theory that the
statement of the AP extends to non-finitely generated members of $\ol\C$ in the following sense.

\begin{lemma}\label{AP}
Let $\C$ be a class of finitely generated structures enjoying the amalgamation property, and let $(A,B,C,f,g)$ be an
amalgam such that $B\in\C$ and $A,C\in\ol\C$. Then $(A,B,C,f,g)$ can be embedded into some structure $D\in\ol\C$.
\end{lemma}

Now suppose that $(A,B,C,f,g)$ is an amalgam in a concrete category $\mathbf{C}$. Then the $\mathbf{C}$-pushout $P$ of
$(A,B,C,f,g)$, if it exists, is usually called the \emph{amalgamated free $\mathbf{C}$-sum} of this amalgam and denoted
by $A\ast_BC$. If, in addition, $A,B,C$ belong to a class $\C$ (contained in $\mathbf{C}$) which has the AP, then
$i_1,i_2$ must be embeddings. Indeed, by the AP there exist a structure $D\in\C$ and embeddings $j_1:A\to D$ and
$j_2:C\to D$; but then there is a homomorphism $u:A\ast_BC\to D$ such that $ui_1=j_1$ and $ui_2=j_2$, forcing $i_1,i_2$
to be embeddings as well.

Here we are going to consider \Fr classes $\C$ satisfying the following \emph{strict amalgamation property}: we require
that for any amalgam $(A,B,C,f,g)$ in $\C$ there exist a structure $P\in\C$ and embeddings $i_1:A\to P$ and $i_2:C\to
P$ such that
\begin{equation}\label{eq.zv}
\hbox{$\vcenter{\xymatrix{%
A \ar@{^{(}->}[r]^{i_1} & P \\
B \ar@{^{(}->}[r]^g \ar@{^{(}->}[u]^{f} & C \ar@{^{(}->}[u]_{i_2} }}$}
\end{equation}
is a pushout square in $\ol\C$, considered as a category of structures and homomorphisms. (Here we follow the
convention that `hooked' arrows always denote embeddings.) In other words, for any amalgam $(A,B,C,f,g)$ in $\C$, their
amalgamated free $\ol\C$-sum exists and belongs to $\C$. The previous lemma can be extended to apply to the strict AP.

\begin{lemma}\label{sAP}
Let $\C$ be a \Fr class enjoying the strict AP. Then for any amalgam $(A,B,C,f,g)$ such that $B\in\C$ and $A,C\in\ol\C$
there exists a $P\in\ol\C$ such that \eqref{eq.zv} is a pushout square in $\ol\C$.
\end{lemma}

\begin{proof}
First of all, consider the case when $B,C\in\C$ and $A\in\ol\C$. Without loss of generality we can assume that $B$ is a
substructure of both $A$ and $C$ and $A\cap C = B$. Let $A = \{a_0, a_1, \ldots\}$ be an enumeration of the elements
of~$A$. Consider the sequence of finitely generated substructures of $A$ defined by $A_0=B$ and $A_{n+1}= \langle A_n
\cup \{a_{m_n}\} \rangle$ for all $n\geq 0$, where $m_n = \min\{i :\ a_i \notin A_n\}$. Clearly, $A_i \in \ol\C$ for
all $i$ since $A \in \ol\C$ and $A = \bigcup_{j<\omega} A_j$.

Now let $P_n$ be a sequence of structures and $h_n:A_n\to P_n$ a sequence of embeddings defined as follows. We set
$P_0=C$ and $h_0=\mathbf{1}_B$ (where throughout the paper $\mathbf{1}_X$ denotes the identity mapping on the set $X$),
that is, $h_0$ is just the inclusion map of $B$ into $C$. For $n\geq 0$, $P_{n+1}$ and $h_{n+1}$ are defined  by the
requirement that the following diagram is a pushout square in $\ol\C$:
$$
\xymatrix{
      P_n \ar@{^{(}->}[r] & P_{n+1} \\
      A_n \ar@{^{(}->}[r]^{\subseteq} \ar@{^{(}->}[u]^{h_n} & A_{n+1} \ar@{^{(}->}[u]_{h_{n+1}} }
$$
where the symbol $\subseteq$ will always denote the corresponding inclusion map. Note that $P_{n+1}$ exists and belongs
to $\C$ by the strict AP. Therefore, we have a sequence of pushout squares:
$$
    \xymatrix{
      {\llap{$C = \mathstrut$}}P_0 \ar@{^{(}->}[r] & P_1 \ar@{^{(}->}[r] & P_2 \ar@{^{(}->}[r] & {\ldots}\\
      {\llap{$B = \mathstrut$}}A_0 \ar@{^{(}->}[r]^{\subseteq} \ar@{^{(}->}[u]^{h_0}
                             & A_1 \ar@{^{(}->}[r]^{\subseteq} \ar@{^{(}->}[u]^{h_1}
                             & A_2 \ar@{^{(}->}[r]^{\subseteq} \ar@{^{(}->}[u]^{h_2} & {\ldots}
    }
$$
Without loss of generality we can assume that $P_n$ is a substructure of $P_{n+1}$ for all $n\geq 0$, so let $P =
\bigcup_{j<\omega} P_j$. Clearly, $P \in \ol\C$ is a union of a chain of structures from~$\C$. Moreover, $h_0 \subseteq
h_1 \subseteq h_2 \subseteq \ldots$ and $i_1 = \bigcup_{j<\omega} h_j$ is an embedding $A \hookrightarrow P$. Let $i_2$
denote the obvious embedding $C = P_0 \hookrightarrow P$. Then the diagram (\ref{eq.zv}) commutes and is easily seen to
be a pushout square in~$\ol\C$.

The statement of the lemma is now extended to the case when $C\in\ol\C$ essentially by repeating an argument analogous
to the above one.
\end{proof}

Let us now quickly review the general method for constructing the \Fr limit of a \Fr class $\C$, cf.\ Hodges \cite{H}.
Recall that a structure $C$ is a \emph{one-point extension} of its substructure $B$ of there is an element $x\in
C\setminus B$ such that $C$ is generated by $B\cup\{x\}$. Trivially, if $B$ is finitely generated, so is $C$.

For a structure $A\in\ol\C$, we first construct its extension $A^\star$. Let
$$
\{(B_i,C_i):\ i<\omega\}
$$
be the enumeration of all pairs of structures such that $B_i\in\C$ is a finitely generated substructure of $A$, while
$C_i$ is a one-point extension of $B_i$ belonging to $\C$; for each isomorphism type we take one such extension. (Here
we consider extensions $(B,C)$ and $(B,C')$ isomorphic if there is an isomorphism $\iota:C\to C'$ such that
$\iota|_B=\mathbf{1}_B$.) Now we construct a chain of structures $A_i$, $i\geq 0$, by successive amalgamations of these
extensions. More precisely, let $A_0=A$ and assume that $A_n$ has already been constructed for some $n\geq 0$ such that
$A\subseteq A_n\in\ol\C$. Then $B_n$ is a substructure of $A$ and so of $A_n$, whence
$(B_n,A_n,C_n,\mathbf{1}_{B_n},\mathbf{1}_{B_n})$ is an amalgam such that $B_n,C_n\in\C$ and $A_n\in\ol\C$. By Lemma
\ref{AP}, there exists a structure $A_{n+1}\in\ol\C$ which embeds this amalgam. As we have just seen, in the case when
$\C$ satisfies the strict AP we can be more specific at this point and let $A_{n+1}=A_n\ast_{B_n}C_n$ (the free sum is
taken with respect to $\ol\C$), by Lemma \ref{sAP}. Clearly, there is no loss of generality in assuming that
$A_n\subseteq A_{n+1}$, so that $A$ is a substructure of $A_{n+1}$. Finally, we let
$$A^\star=\bigcup_{n<\omega}A_n.$$

Of course, this construction can be iterated, so that we apply $\aleph_0$ successive rounds of amalgamation. Namely,
let $A^{(0)}=A$ and define $A^{(n+1)}=(A^{(n)})^\star$ for all $n\geq 0$. We set
$$F_A=\bigcup_{n<\omega}A^{(n)},$$
which is an extension of $A$. Any finitely generated substructure of $F_A$ must belong to some $A^{(m)}$, and since
$A^{(m)}\in\ol\C$ the finitely generated structure in question belongs to $\C$; hence, $F_A\in\ol\C$. A standard
model-theoretic argument (which can be found in \cite{D-Berg}) shows that $F_A$ is precisely the \Fr limit of $\C$.

\begin{proposition}\label{flim}
For any $A\in\ol\C$ we have $F_A\cong\Fl(\C)$.
\end{proposition}

Now we would like to modify the above construction, so that instead of amalgamating the extensions $C_i$ one by one
into $A^\star$ we do it ``all at once''. To this end, suppose we have a structure $A$ and a family of its substructures
$\{B_i:\ i\in I\}$, while $\{C_i:\ i\in I\}$ is another family of structures such that $A\cap C_i=B_i$ for each $i\in
I$ (so that $B_i$ is a substructure of $C_i$). Furthermore, for any $i,j\in I$, $i\neq j$, we assume that
$(C_i\setminus B_i)\cap(C_j\setminus B_j)=\es$. Then the tuple
$$(A,(B_i,C_i)_{i\in I})$$
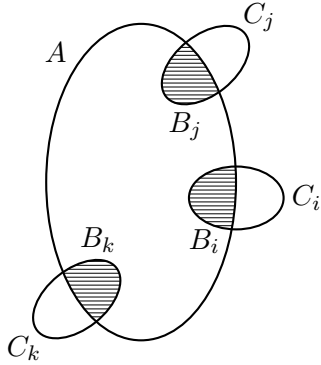
is called a \emph{rooted multi-amalgam}. Assume $\C$ is a class of finitely generated structures of a fixed signature.
If we are concerned with a rooted multi-amalgam in which $A\in\ol\C$ is a countable structure, each $B_i$ is a finitely
generated substructure of $A$, $B_i\in\C$, and each $C_i$ is a finitely generated extension of $B_i$ lying in $\C$,
then we call $(A,(B_i,C_i)_{i\in I})$ a rooted multi-amalgam \emph{over} $\C$, and the structure $A$ is its
\emph{root}.

\smallskip

\begin{center}
\psset{unit=0.24pt}
\begin{pspicture}(0.0,0.0)(550.0,550.0)
    \radians
    \pscustom{
        \msave
        \translate(400.0,250.0)
        \rotate{3.141593}
        \psellipse[linewidth=0.07mm](0.0,0.0)(75.0,50.0)
        \mrestore
    }
    \pscustom{
        \msave
        \translate(351.451,459.568)
        \rotate{3.816334}
        \psellipse[linewidth=0.07mm](0.0,0.0)(80.0391,47.4186)
        \mrestore
    }
    \pscustom{
        \msave
        \translate(148.549,90.4321)
        \rotate{0.674741}
        \psellipse[linewidth=0.07mm](0.0,0.0)(80.0391,47.4186)
        \mrestore
    }
    \psline[linewidth=0.05mm](318.39,493.75)(322.618,493.75)
    \psline[linewidth=0.05mm](311.577,487.5)(329.017,487.5)
    \psline[linewidth=0.05mm](305.634,481.25)(334.77,481.25)
    \psline[linewidth=0.05mm](300.441,475.0)(340.0,475.0)
    \psline[linewidth=0.05mm](295.928,468.75)(344.794,468.75)
    \psline[linewidth=0.05mm](292.054,462.5)(349.216,462.5)
    \psline[linewidth=0.05mm](288.798,456.25)(353.312,456.25)
    \psline[linewidth=0.05mm](286.162,450.0)(357.121,450.0)
    \psline[linewidth=0.05mm](284.167,443.75)(360.673,443.75)
    \psline[linewidth=0.05mm](282.854,437.5)(363.99,437.5)
    \psline[linewidth=0.05mm](282.298,431.25)(367.094,431.25)
    \psline[linewidth=0.05mm](282.62,425.0)(370.0,425.0)
    \psline[linewidth=0.05mm](284.026,418.75)(372.723,418.75)
    \psline[linewidth=0.05mm](286.896,412.5)(366.684,412.5)
    \psline[linewidth=0.05mm](292.077,406.25)(354.955,406.25)
    \psline[linewidth=0.05mm](302.628,400.0)(337.854,400.0)
    \psline[linewidth=0.05mm](363.691,293.75)(399.578,293.75)
    \psline[linewidth=0.05mm](350.392,287.5)(399.812,287.5)
    \psline[linewidth=0.05mm](341.453,281.25)(399.953,281.25)
    \psline[linewidth=0.05mm](335.095,275.0)(400.0,275.0)
    \psline[linewidth=0.05mm](330.473,268.75)(399.953,268.75)
    \psline[linewidth=0.05mm](327.382,262.5)(399.812,262.5)
    \psline[linewidth=0.05mm](325.588,256.25)(399.578,256.25)
    \psline[linewidth=0.05mm](325.0,250.0)(399.248,250.0)
    \psline[linewidth=0.05mm](325.588,243.75)(398.824,243.75)
    \psline[linewidth=0.05mm](327.382,237.5)(398.303,237.5)
    \psline[linewidth=0.05mm](330.473,231.25)(397.685,231.25)
    \psline[linewidth=0.05mm](335.048,225.0)(396.969,225.0)
    \psline[linewidth=0.05mm](341.453,218.75)(396.154,218.75)
    \psline[linewidth=0.05mm](350.392,212.5)(395.237,212.5)
    \psline[linewidth=0.05mm](363.691,206.25)(394.217,206.25)
    \psline[linewidth=0.05mm](162.146,150.0)(197.372,150.0)
    \psline[linewidth=0.05mm](145.045,143.75)(207.923,143.75)
    \psline[linewidth=0.05mm](133.316,137.5)(213.104,137.5)
    \psline[linewidth=0.05mm](127.277,131.25)(215.974,131.25)
    \psline[linewidth=0.05mm](130.0,125.0)(217.38,125.0)
    \psline[linewidth=0.05mm](132.906,118.75)(217.702,118.75)
    \psline[linewidth=0.05mm](136.01,112.5)(217.146,112.5)
    \psline[linewidth=0.05mm](139.327,106.25)(215.833,106.25)
    \psline[linewidth=0.05mm](142.879,100.0)(213.838,100.0)
    \psline[linewidth=0.05mm](146.688,93.75)(211.202,93.75)
    \psline[linewidth=0.05mm](150.784,87.5)(207.946,87.5)
    \psline[linewidth=0.05mm](155.206,81.25)(204.072,81.25)
    \psline[linewidth=0.05mm](160.0,75.0)(199.559,75.0)
    \psline[linewidth=0.05mm](165.23,68.75)(194.366,68.75)
    \psline[linewidth=0.05mm](170.983,62.5)(188.423,62.5)
    \psline[linewidth=0.05mm](177.382,56.25)(181.61,56.25)
    \pscustom{
        \msave
        \translate(250.0,275.0)
        \rotate{1.570796}
        \psellipse[linewidth=0.07mm](0.0,0.0)(250.0,150.0)
        \mrestore
    }
    \put(115.097,479.095){\makebox(0,0){$A$}}
    \put(349.538,201.01){\makebox(0,0)[t]{$B_i$}}
    \put(319.573,385.354){\makebox(0,0)[t]{$B_j$}}
    \put(181.914,164.645){\makebox(0,0)[b]{$B_k$}}
    \put(486.999,250.269){\makebox(0,0)[l]{$C_i$}}
    \put(411.05,512.585){\makebox(0,0)[bl]{$C_j$}}
    \put(91.8887,32.5775){\makebox(0,0)[tr]{$C_k$}}
\end{pspicture}

\smallskip

{\small \textsl{Figure 1.1.} A rooted multi-amalgam}
\end{center}

\smallskip

The \emph{free $\C$-sum} of the rooted amalgam $(A,(B_i,C_i)_{i\in I})$ over $\C$ is a structure $D\in\ol\C$ with the
following properties:
\begin{itemize}
\item[(a)] there are embeddings $f:A\to D$ and $g_i:C_i\to D$, $i\in I$, such that $f|_{B_i}=g_i|_{B_i}$ for any $i\in
I$,
\item[(b)] for any structure $D'\in\ol\C$ and any homomorphisms $\varphi:A\to D'$, $\psi_i:C_i\to D'$, $i\in I$, such that
for any $i\in I$ we have $\varphi|_{B_i}=\psi_i|_{B_i}$, there exists a unique homomorphism $\delta:D\to D'$ extending
all the given homomorphisms, that is, such that we have  $\delta f=\varphi$ and $\delta g_i=\psi_i$ for all $i\in I$.
\end{itemize}
Speaking just a bit more loosely, the free sum is freely generated within $\ol\C$ by its partial substructure
$A\cup\bigcup_{i\in I}C_i$. Yet another way of saying this is that the free $\C$-sum is the colimit of the following
diagram:
\begin{equation}\label{eq.excl}
  \hbox{$\vcenter{
  \xymatrix{
        &   A \\
    B_i \ar@{_{(}->}[d]_{\subseteq} \ar@{^{(}->}[ur]^{\subseteq} &
    B_j \ar@{_{(}->}[d]_{\subseteq} \ar@{^{(}->}[u]^{\subseteq}  &
    B_k \ar@{_{(}->}[d]^{\subseteq} \ar@{^{(}->}[ul]_{\subseteq} &  {\ldots}\\
    C_i &   C_j   &  C_k  & {\ldots}
  }
  }$}
\end{equation}
The free $\C$-sum $D$ of $(A,(B_i,C_i)_{i\in I})$, if it exists, is unique up to isomorphism and generated by
$f(A)\cup\bigcup_{i\in I}g_i(C_i)$. We denote it by $\coprod\st(A,(B_i,C_i)_{i\in I})$.

\begin{lemma}\label{fsum}
Let $\C$ be a \Fr class enjoying the strict AP. Then for every rooted multi-amalgam $(A,(B_i,C_i)_{i\in I})$ over $\C$
the free $\C$-sum $\coprod\st(A, (B_i, C_i)_{i \in I})$ exists and belongs to $\ol\C$.
\end{lemma}

\begin{proof}
Since each $B_i, C_i$ is finitely generated, the index set $I$ is countable, so there is no loss of generality if we
assume that $I$ is in fact the ordinal $\omega$. Let us inductively define a sequence of structures $P_i$, $i\geq 0$,
as follows. First of all, let $P_0=A$. Given $P_n$, let $P_{n+1}$ be the ($\ol\C$-)pushout of the amalgam
$(C_n,B_n,P_n,\mathbf{1}_{B_n},\mathbf{1}_{B_n})$ (note that $B_n$ is a substructure of $A=P_0$ and so of each $P_n$).
This pushout exists and belongs to $\ol\C$ by Lemma \ref{sAP}. Also, there will be no loss of generality in assuming
that $P_n$ is actually a substructure of $P_{n+1}$ for each $n\geq 0$.

Let $P =\bigcup_{j<\omega} P_j$. Clearly, $P \in \ol\C$, since $P$ is the union of a chain of structures from $\ol\C$.
Let us show that $P$ satisfies the properties required by the free $\C$-sum of $(A,(B_i,C_i)_{i\in I})$, i.e.\ that it
is the colimit of the diagram \eqref{eq.excl} in $\ol\C$. Obviously, the condition (a) from the definition of a free
$\C$-sum is satisfied, as both $A$ and all $C_i$'s (in fact, the whole rooted multi-amalgam $(A,(B_i,C_i)_{i<\omega})$)
are contained in $P$, so the corresponding inclusion mappings will take the role of $f$ and the $g_i$'s.

\begin{figure}[t]
\centering
$$
  \begin{array}{ccc}
    \xymatrix{
                          &     & Q\\
      C_0 \ar@{^{(}->}[r]^\subseteq \ar@/^2ex/[urr]^{\psi_0} & P_1 \ar[ur]^{m_1}\\
      B_0 \ar@{^{(}->}[r]^\subseteq \ar@{^{(}->}[u]^\subseteq & A \ar@{^{(}->}[u]^\subseteq \ar@/_1ex/[uur]_{\varphi}
    }
  &
    \xymatrix{
      C_1 \ar[rr]^{\psi_1} & & Q\\
      B_1 \ar@{^{(}->}[r]^\subseteq \ar@{^{(}->}[u]^\subseteq & A \ar[ur]^\varphi \ar@{^{(}->}[r]^\subseteq &
      P_1 \ar[u]_{m_1}
    }
  &
    \xymatrix{
                               &          & Q\\
      C_n \ar@{^{(}->}[r]^\subseteq \ar@/^3ex/[urr]^{\psi_n} & P_{n+1} \ar[ur]^{m_{n+1}} \\
      B_n \ar@{^{(}->}[r]^\subseteq \ar@{^{(}->}[u]^\subseteq & P_n \ar@{^{(}->}[u]^\subseteq \ar@/_1ex/[uur]_{m_n}
    }
  \\
    (a) & (b) & (c)
  \end{array}
$$
{\small \textsl{Figure 1.2.} Three commuting diagrams}
\end{figure}

Now let $Q\in\ol\C$ and let $\varphi:A\to Q$, $\psi_i:C_i\to Q$, $i\geq 0$, be homomorphisms such that
$\varphi|_{B_i}=\psi_i|_{B_i}$ for all $i$. By the construction of $P_1$, there is a unique homomorphism $m_1 : P_1 \to
Q$ such that the diagram in Fig.\ 1.2~$(a)$ commutes. The diagram in Fig.\ 1.2~$(b)$ also commutes: the triangle
commutes by Fig.\ 1.2~$(a)$, while the square commutes by the initial assumption on $Q$. Hence, the outer square in the
diagram in Fig.\ 1.2~$(c)$ commutes for $n = 1$, so there is a unique homomorphism $m_2 : P_2 \to Q$ that makes the
entire diagram commutative for $n = 1$.  By induction, for every $n \geq 1$, there is a unique homomorphism $m_{n+1} :
P_{n+1} \to Q$ such that the diagram in Fig.\ 1.2~$(c)$ commutes. Thus we have a sequence of homomorphisms $\varphi
\subseteq m_1 \subseteq m_2 \ldots$, so define $m=\bigcup_{j<\omega}m_j$. Clearly, $m : P \to Q$ and the diagram in
Fig.\ 1.3~$(a)$ commutes or all $i\geq 0$.

\begin{figure}[h]
\centering
$$
  \begin{array}{ccc}
  \xymatrix{
     A \ar[rr]^\varphi  \ar@/_2ex/@{^{(}->}[ddrr]^\subseteq& & Q\\
     B_i \ar@{_{(}->}[d]_\subseteq \ar@{^{(}->}[u]^\subseteq \\
     C_i \ar@/^2ex/[uurr]_(0.7){\psi_i} \ar@{^{(}->}[rr]^\subseteq & & P \ar[uu]_m
  }
  &
  \xymatrix{
     A \ar[rr]^\varphi  \ar@/_2ex/@{^{(}->}[ddrr]^\subseteq& & Q\\
     B_1 \ar@{_{(}->}[d]_\subseteq \ar@{^{(}->}[u]^\subseteq \\
     C_1 \ar@/^2ex/[uurr]_(0.7){\psi_1} \ar@{^{(}->}[rr]^\subseteq & & P_1 \ar[uu]_{m'_1}
  }
  &
    \xymatrix{
                               &           & Q\\
      C_2 \ar@{^{(}->}[r]^\subseteq \ar@/^3ex/[urr]^{\psi_2}  & P_2 \ar[ur]^{m'_2} \\
      B_2 \ar@{^{(}->}[r]^\subseteq \ar@{^{(}->}[u]^\subseteq & P_1 \ar@{^{(}->}[u]^\subseteq \ar@/_1ex/[uur]_{m'_1}
    }
  \\
  (a) & (b) & (c)
  \end{array}
$$
{\small \textsl{Figure 1.3.} The mediating homomorphism}
\end{figure}

Let $m' : P \to Q$ be another homomorphism that makes the diagram in Fig.\ 1.3~$(a)$ commutative for all $i\geq 0$, and
let us show that $m' = m$. Let $m'_n = m'|_{P_n} : P_n \to Q$ for $n \geq 1$, and let $m'_0 = m'|_A$. Clearly, $m'_0 =
\varphi$ (see Fig.\ 1.3~$(a)$). Specializing to $i = 0$ and using the fact that $C_0$ and $A$ are substructures of
$P_1$, we get the diagram in Fig.\ 1.3~$(b)$. Therefore, $m'_1 = m_1$ since $m_1$ is the unique homomorphism $P_1 \to
Q$ that makes the diagram commutative. On the other hand, the diagram in Fig.\ 1.3~$(c)$ commutes (the two triangles
commute because of the assumptions on $m'$ and the fact that $m'_1$ is a restriction of $m'_2$ and $m'_2$ is a
restriction of $m'$). Therefore, $m'_2 = m_2$ since $m_2$ is the unique homomorphism $P_2 \to Q$ that makes the diagram
commutative. Proceeding by induction we obtain $m'_n=m_n$ for every $n \geq 1$, so $m' = m$. This completes the proof
that $P$ is the colimit of the diagram \eqref{eq.excl} in $\ol\C$, i.e.\ that $P$ is the free $\C$-sum of
$(A,(B_i,C_i)_{i\in I})$.
\end{proof}

\begin{example}\rm\label{examp}
\begin{enumerate}
\item Finite simple graphs have the strict AP, since the free sum of an amalgam of finite graphs is basically the
amalgam itself. So, the free sum of a rooted multi-amalgam $(A,(B_i,C_i)_{i\in I})$ over the class of finite simple
graphs is the graph with the vertex set $A\cup\bigcup_{i<\omega}C_i$ in which the set of edges is just the union of
edges of $A$ and $C_i$'s.
\item Finite posets have the strict AP: consequently, the free sum $\coprod\st(A,(B_i,C_i)_{i\in I})$ of a rooted
multi-amalgam over the class of finite posets is just the transitive closure of the reflexive and antisymmetric
relation on $A\cup\bigcup_{i<\omega}C_i$ obtained as the union of order relations on $A$ and $C_i$'s.
\item As is well-known (see \cite{Gr}), if $\mathbf{V}$ if a variety of algebras, then $\coprod\st(A,(B_i,C_i)_{i\in I})$
is just the free algebra freely generated in $\mathbf{V}$ by the partial algebra $(A,(B_i,C_i)_{i\in I})$, which exists
if and only if the rooted multi-amalgam $(A,(B_i,C_i)_{i\in I})$ can be embedded into a member of $\mathbf{V}$;
essentially, for algebraic structures the ordinary and the strict AP for the class of finitely generated members of
$\mathbf{V}$ are equivalent (bearing in mind Lemmata \ref{AP} and \ref{sAP}). Therefore, finite semilattices, finite
distributive lattices and finite Boolean algebras, respectively, have the strict AP and thus all free sums of their
rooted multi-amalgams exist in the corresponding varieties.
\end{enumerate}
\end{example}

Given the notion of a rooted multi-amalgam and the conclusion of Lemma \ref{fsum}, we change a bit the construction of
the extension $A^\star$ of a structure $A\in\ol\C$. As before, let $\{(B_i,C_i):\ i<\omega\}$ be an arbitrary
enumeration of all pairs where $B_i$ is a finitely generated substructure of $A$ and $C_i$ is a one-point extension of
$B_i$, such that each isomorphism type of such pairs is represented exactly once. Then, clearly, $(A,(B_i,C_i)_{i\in
I})$ is a rooted multi-amalgam over $\C$, so we redefine
$$
\textstyle A^\star=\coprod\st(A,(B_i,C_i)_{i\in I}).
$$
It is not difficult to show that if we iterate this construction and define
$$F\st_A=\bigcup_{n<\omega}A^{(n)},$$
then the result is again isomorphic to the \Fr limit of $\C$. In other words, we have the following variant of the
previous proposition.

\begin{proposition}\label{flim2}
For any $A\in\ol\C$ we have $F\st_A\cong\Fl(\C)$.
\end{proposition}

\begin{proof}
By \cite[Lemma 6.1.3]{H}, since $F\st_A$ is countably infinite, it suffices to show that it is \emph{weakly
homogeneous} (or that it \emph{realizes all one-point extensions} in the terminology of \cite{DK}), i.e.\ that for each
finitely generated substructure $B$ of $F\st_A$ and its one-point extension $C\in\C$ there should be an embedding
$f:C\to F\st_A$ such that $f|_B=\mathbf{1}_B$.

Indeed, since $B$ is finitely generated, there exists an index $k<\omega$ such that $B\subseteq A^{(k)}$. Furthermore,
if $\{(B_i^{(k)},C_i^{(k)}):\ i<\omega\}$ is the enumeration of all isomorphism types of finitely generated
substructures of $A^{(k)}$ and its one-point extensions, yielding the components of the rooted multi-amalgam the free
sum of which is $A^{(k+1)}$, then for some $p<\omega$ we have $B=B_p^{(k)}$ and $C\cong C_p^{(k)}$. By the properties
of the free $\C$-sum, the latter isomorphism gives rise to an embedding of $C$ into $A^{(k+1)}$, and thus into
$F\st_A$.
\end{proof}

\section{A general embedding theorem}

A fundamental property of the free $\C$-sum is that in order to extend an endomorphism $\varphi$ of a structure $A$ to
an endomorphism $\wh\varphi$ of $A^\star$, one only needs to define homomorphisms $\psi_i:C_i\to A^\star$, $i<\omega$,
which agree with $\varphi$ on each $B_i$. This underlines the importance of the following configuration: we have a
finitely generated structure $B\in\C$ and a surjective homomorphism $f:B\to B'$, where $B'\in\C$; at the same time, we
have a one-point extension $C$ of $B$. In fact, we have a span
$$
\xymatrix{C & B \ar@{_{(}->}[l]_{\mathbf{1}_B} \ar@{->>}[r]^f & B'}
$$
where we follow the convention that `double-headed' arrows stand for surjective homomorphisms (recall that `hooked'
ones denote embeddings).

We say that a class $\C$ enjoys the \emph{one-point homomorphism extension property} (1PHEP) if for any $B,B',C\in\C$
forming a configuration as above there exists an extension $C'$ of $B'$ and a surjective homomorphism $f\st:C\to C'$
such that $f\st|_B=f$; in other words, the following diagram commutes:
$$
\xymatrix{C \ar@{->>}[r]^{f\st} & C'\\
B \ar@{->>}[r]^f \ar@{^{(}->}[u]^{\mathbf{1}_B} & B' \ar@{^{(}->}[u]_{\mathbf{1}_{B'}}}
$$
As it easily turns out, $C'$ is either a one-point extension of $B'$, or it coincides with $B'$. As noticed in Remark
3.1 of \cite{D-Berg}, this property of a \Fr class $\C$ is equivalent to the \emph{homomorphism amalgamation property}
(HAP) intimately related to homomorphism-homogeneity. In fact, by Proposition 3.8 of \cite{D-Berg}, the
ultrahomogeneous structure $\Fl(\C)$ is homomorphism-homoge\-ne\-ous if and only if $\C$ has the 1PHEP. Results of
\cite{D-Berg} show that 1PHEP is enjoyed by the classes of finite simple graphs, finite posets, finite metric spaces
(with rational distances), as well as by each class of finitely generated algebras with the \emph{congruence extension
property} (CEP) \cite{Gr} closed under taking homomorphic images, such as finite semilattices.

We need a stronger form of the 1PHEP for our purpose, related to a construction that in a way represents the analogue
of the amalgamated free sum for the spans considered here. We say that a class of structures $\C$ has the \emph{strict
1PHEP} if any span of the form
$$
\xymatrix{C & B \ar@{_{(}->}[l]_i \ar@{->>}[r]^f & B'}
$$
where $B,B',C\in\C$ and $C$ is a one-point extension of $B$, has a pushout $P\in\C$ with respect to $\ol\C$ as a
category (of structures and homomorphisms), and if
$$
\xymatrix{C \ar[r]^{f'} & P\\
B \ar@{->>}[r]^f \ar@{^{(}->}[u]^i & B' \ar[u]_{i'}}
$$
is a pushout square in $\ol\C$, then $i'$ is an embedding and the homomorphism $f'$ is surjective. In other words, the
pushout of the above span serves as a witness in the particular instance of the 1PHEP for the span in question.

The following is the main result of this paper.

\begin{theorem}\label{main}
Let $\C$ be a \Fr class satisfying the following three properties:
\begin{itemize}
\item[(i)] $\C$ enjoys the strict AP.
\item[(ii)] $\C$ enjoys the strict 1PHEP.
\item[(iii)] For any $B,C\in\C$ such that $C$ is a one-point extension of $B$, the pointwise stabilizer of $B$ in
$\Aut(C)$ is trivial.
\end{itemize}
Then for any $A\in\ol\C$ there is an embedding of $\End(A)$ into $\End(A^\star)$. Consequently, if $F=\Fl(\C)$, then
$\End(A)$ embeds into $\End(F)$.
\end{theorem}

\begin{proof}
We start by considering an arbitrary $\varphi\in\End(A)$. Assume that $A^\star$ is obtained as a free $\C$-sum of the
rooted multi-amalgam $(A,(B_i,C_i)_{i\in I})$. Our first task here is to specify homomorphisms
$\psi_i^{(\varphi)}:C_i\to A^\star$, $i<\omega$, agreeing with $\varphi$ on each $B_i$. For convenience, write
$\varphi_i=\varphi|_{B_i}$. Then there is an index $j<\omega$ such that $\varphi(B_i)=B_j$. By (ii) we have the pushout
square
$$
\xymatrix{C_i \ar@{->>}[r]^{\xi_i^{(\varphi)}} & C'_i \\
B_i \ar@{->>}[r]_{\varphi_i} \ar@{^{(}->}[u]_\subseteq & B_j \ar@{^{(}->}[u]_\subseteq }
$$
 However, $B_j$ has come with the accompanying
$C_j$ such that the extension $(B_j, C_j)$ is of the same isomorphism type as $(B_j, C'_i)$. Therefore, there is an
isomorphism $\iota_i^{(\varphi)} : C'_i \to C_j$ such that
$$
\xymatrix{C_i \ar@{->>}[r]^{\xi_i^{(\varphi)}} & C'_i \ar[r]^{\iota_i^{(\varphi)}} & C_j\\
B_i \ar@{->>}[rr]_{\varphi_i} \ar@{^{(}->}[u]_\subseteq & & B_j \ar@{^{(}->}[ul]_\subseteq \ar@{^{(}->}[u]_\subseteq }
$$
Since $B_i B_j C_i C'_i$ is a pushout square and $C_j$ is isomorphic to $C'_i$, it follows that $B_i B_j C_i C_j$ is
also a pushout square. Now we define $\psi_i^{(\varphi)}=\iota_i^{(\varphi)}\xi_i^{(\varphi)}$; therefore, we have a
family of homomorphisms
$$
(\varphi,(\psi_i^{(\varphi)})_{i<\omega})
$$
defined on the components of the rooted multi-amalgam in $\C$ whose sum is $A^\star$. By the properties of the free
$\C$-sum, this family extends to a unique endomorphism $\wh\varphi$ of $A^\star$. We claim that the assignment
$\varphi\mapsto\wh\varphi$ is a monoid embedding $\End(A)\to\End(A^\star)$. Since $\wh\varphi|_A=\varphi$ the latter
assignment is injective, and since $(\mathbf{1}_A,(\mathbf{1}_{C_i})_{i<\omega})$ yields
$\wh{\mathbf{1}_A}=\mathbf{1}_{A^\star}$, we need to prove that for any two $\varphi,\varphi'\in\End(A)$ we have
$$\wh{\varphi'\varphi}=\wh{\varphi'}\wh\varphi.$$

To this end, let $\varphi'$ be another endomorphism of $A$, $\Phi=\varphi'\varphi$ and $\varphi'_i=\varphi'|_{B_i}$,
$\Phi_i=\Phi|_{B_i}$ for any $i<\omega$. We still assume that the index $j$ is selected for a given arbitrary $i$ such
that $\varphi(B_i)=B_j$. So, let $\phi'_j(B_j) = B_k$. By the definition of $\wh{\varphi'}$ there exist a surjective
homomorphism $\xi_j^{(\varphi')} : C_j \to C'_j$ and an isomorphism $\iota_j^{(\varphi')} : C'_j \to C_k$, where $(B_k,
C_k)$ is the representative of the isomorphism type of $(B_k, C'_j)$, such that $B_j B_k C_j C_k$ in the following
diagram
$$
\xymatrix{C_i \ar@{->>}[r]^{\xi_i^{(\varphi)}} & C'_i \ar[r]^{\iota_i^{(\varphi)}} & C_j
\ar@{->>}[r]^{\xi_j^{(\varphi')}} & C'_j \ar[r]^{\iota_j^{(\varphi')}} & C_k\\
B_i \ar@{->>}[rr]_{\phi_i} \ar@{^{(}->}[u]_\subseteq & & B_j \ar@{^{(}->}[ul]_\subseteq \ar@{^{(}->}[u]_\subseteq
\ar@{->>}[rr]_{\phi'_j} & & B_k \ar@{^{(}->}[ul]_\subseteq \ar@{^{(}->}[u]_\subseteq }
$$
is a pushout square. Then, by the properties of pushouts, $B_i B_k C_i C_k$ is a pushout square. On the other hand, we
can repeat the procedure in a slightly different setting to get the following pushout square
$$
\xymatrix{C_i \ar@{->>}[r]^{\xi_i^{(\varphi'\varphi)}} & C''_i \ar[r]^{\iota_i^{(\varphi'\varphi)}} & C_k\\
B_i \ar@{->>}[rr]_{\Phi_i = \varphi'_j \varphi_i} \ar@{^{(}->}[u]_\subseteq & & B_k \ar@{^{(}->}[ul]_\subseteq
\ar@{^{(}->}[u]_\subseteq }
$$
Therefore, there exist homomorphisms $m, n : C_k \to C_k$ such that the diagram below commutes (both inner and the
outer square are pushout squares):
$$
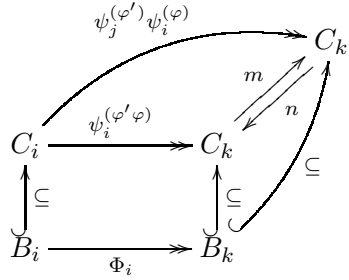
\xymatrix{& & & C_k \ar@<0.5ex>[dl]^n \\
C_i \ar@{->>}[rr]^{\psi_i^{(\varphi'\varphi)}} \ar@{->>}@/^4ex/[urrr]^{\psi_j^{(\varphi')}
\psi_i^{(\varphi)}} & & C_k \ar@<0.5ex>[ur]^m & \\
B_i \ar@{->>}[rr]_{\Phi_i} \ar@{^{(}->}[u]_\subseteq & & B_k \ar@{^{(}->}[u]_\subseteq
\ar@{^{(}->}@/_2ex/[ruu]_\subseteq & }
$$
In particular,
\begin{gather*}
m\psi_i^{(\varphi'\varphi)} = \psi_j^{(\varphi')}\psi_i^{(\varphi)},\\
n(\psi_j^{(\varphi')}\psi_i^{(\varphi)}) = \psi_i^{(\varphi'\varphi)}.
\end{gather*}
Then
\begin{gather*}
nm\psi_i^{(\varphi'\varphi)} = n(\psi_j^{(\varphi')}\psi_i^{(\varphi)})=\psi_i^{(\varphi'\varphi)},\\
mn(\psi_j^{(\varphi')}\psi_i^{(\varphi)}) = m\psi_i^{(\varphi'\varphi)}=\psi_j^{(\varphi')}\psi_i^{(\varphi)}.
\end{gather*}
Since both $\psi_j^{(\varphi')}\psi_i^{(\varphi)}$ and $\psi_i^{(\varphi'\varphi)}$ are surjective, it follows that
$mn=nm=\mathbf{1}_{C_k}$, so $m$ and $n$ are mutually inverse automorphisms of $C_k$ fixing $B_k$ pointwise. The
requirement (iii) now yields that $m = n = \mathbf{1}_{C_k}$, whence
$$
\psi_j^{(\varphi')}\psi_i^{(\varphi)} = \psi_i^{(\varphi'\varphi)}.
$$
This completes the proof that $\wh{\varphi'\varphi} = \wh{\varphi'}\wh\varphi$.
\end{proof}

We say that a one-point extension $C$ of $B$ is \emph{uniquely generated} if for any $x,x'\in C\setminus B$ such that
$C=\langle B\cup\{x\}\rangle = \langle B\cup\{x'\}\rangle$ we have $x=x'$. We focus our interest on \Fr classes with
this property, since it is always present in (but, as we shall see, not confined to) the case when the signature of
$\C$ contains no operation symbols (in particular, in all relational structures): the unique generator is the single
element in $C\setminus B$. Notice that the uniquely generated one-point extensions instantly satisfy condition (iii)
from the previous theorem, so that it may be dropped altogether e.g.\ for classes of relational structures.

\begin{corollary}\label{ug}
Let $\C$ be a \Fr class satisfying the condition of uniquely generated one-point extensions. If $\C$ satisfies the
strict AP and the strict 1PHEP, then for any $A\in\ol\C$, $\End(A)$ embeds into $\End(A^\star)$ and so into the
endomorphism monoid of $\Fl(\C)$.
\end{corollary}

\section{Applications}

\subsection{The cases of the random graph and the generic poset}

As noted in items 1 and 2 of Example \ref{examp}, the classes of finite simple graphs and finite posets enjoy the
strict AP and thus admit free sums of their rooted multi-amalgams. Of course, these classes have the property of
uniquely generated one-point extensions, thus we may use Corollary \ref{ug}. Our aim in this section is to briefly
reprove the main results of the notes \cite{BDD} and \cite{D} in order to show that they are transparent instances of
our general scheme.

\begin{lemma}\label{R-fun}
The class of finite simple graphs satisfies the strict 1PHEP.
\end{lemma}

\begin{proof}
Let $G$ be a finite simple graph, $f:G\to G'$ a surjective graph homomorphism and $H=G\cup\{\vv\}$ a one-point
extension of $G$ (for convenience, we assume that the embedding $i$ in the definition of the strict 1PHEP is an
inclusion map). Let $N_G(\vv)\subseteq V(G)$ be the set of all vertices adjacent to $\vv$. Define $H'$ to be the graph
obtained from $G'$ by adjoining a new vertex $\vv'$ which is adjacent to all vertices from $\{f(\ww):\ \ww\in
N_G(\vv)\}$ and no other vertex from $V(G')$. Furthermore, let $f'$ be the homomorphism extending $f$ which sends $\vv$
to $\vv'$.

We claim that $H'$, along with $\mathbf{1}_{G'}$ and $f'$, is a pushout for the span $(G,G',H,f,\mathbf{1}_G)$. Indeed,
let $\Gamma$ be a graph such that there are homomorphisms $g:G'\to \Gamma$ and $h:H\to\Gamma$ with $gf=h|_G$. Now there
is only one function $u:V(H')\to V(\Gamma)$ such that $u|_{V(G')}=g$ and $uf'=h$: it is defined by
$u(\mathbf{x})=g(\mathbf{x})$ for $\mathbf{x}\in V(G')$ and $u(\vv')=h(\vv)$. Since $g$ is a graph homomorphism, and
since for any $\ww\in N_G(\vv)$ we have that $u(f'(\ww))=uf(\ww)=h(\ww)$ is adjacent to $h(\vv)=u(\vv')$ (as $h$ is a
homomorphism, too), it follows that $u$ is a graph homomorphism as well. Since $f'$ is surjective, the lemma is proved.
\end{proof}

\begin{corollary}[\cite{BDD}]\label{dm}
$\Self(\aleph_0)$ (and so any countable semigroup) embeds into $\End(R)$, the endomorphism monoid of the countably
infinite random graph.
\end{corollary}

\begin{proof}
In Corollary \ref{ug} choose $A$ to be the countably infinite graph with no edges, since then $\End(A)\cong
\Self(\aleph_0)$.
\end{proof}

\begin{lemma}\label{P-fun}
The class of finite posets satisfies the strict 1PHEP.
\end{lemma}

\begin{proof}
Let $P$ be a finite poset, $f:P\to P'$ a surjective poset homomorphism (order-preserving map) and $Q=P\cup\{\xx\}$ a
one-point extension of $P$. In order to define the poset $Q'$, let
$$
L=\{p\in P:\ p<\xx\}\mbox{\quad and\quad}U=\{p\in P:\ \xx<p\}.
$$
Clearly, $\ell<u$ for any $\ell\in L$ and $u\in U$; therefore $f(\ell)\leq f(u)$ holds in $P'$. We distinguish two
cases. If $f(L)\cap f(U)=\es$, then we define $Q'$ to be the poset obtained from $P'$ by inserting a new element $\xx'$
`between' the down-set $f(L)$ and up-set $f(U)$: we have $y<\xx'$ if and only if $y\in f(L)$ and $\xx'<y$ if and only
if $y\in f(U)$ (the other elements of $P'$ are set to be incomparable to $\xx'$). Then, the homomorphism $f':Q\to Q'$
is defined as the extension of $f$ obtained by sending $\xx\mapsto\xx'$. However, if $f(L)\cap f(U)\neq\es$, then for
any $y\in f(L)\cap f(U)$, $\ell'\in f(L)$ and $u'\in f(U)$ we have $\ell'\leq y\leq u'$, implying that such $y$ must be
unique: we have $f(L)\cap f(U)=\{y_0\}$ for some $y_0\in P'$. Then we define $Q'=P'$ and $f':Q\to P'$ is an extension
of $f$ determined by $f(\xx)=y_0$. In both cases $f'$ is surjective.

We claim that $Q'$ is the pushout of the span $(P,P'Q,f,\mathbf{1}_P)$, together with the poset homomorphisms
(order-preserving maps) $\mathbf{1}_{P'}$ and $f'$; that is, we claim that there is a unique order-preserving map
$m:Q'\to X$ making the following diagram commutative:
$$
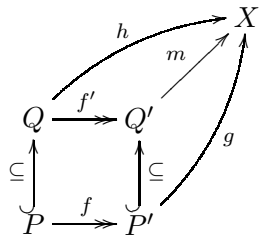
\xymatrix{&& X\\
Q \ar@{->>}[r]^{f'} \ar@/^2ex/[urr]^{h} & Q' \ar[ru]^m & \\
P \ar@{->>}[r]^f \ar@{^{(}->}[u]^\subseteq & P' \ar@{^{(}->}[u]_\subseteq \ar@/_2ex/[ruu]_{g} & }
$$
where $X$ is a poset, while $g:P'\to X$ and $h:Q\to X$ are order-preserving functions such that $gf=h|_P$. As in the
case of graphs, there is a single function $m:Q'\to X$ that makes the above diagram commutative, and it is defined by
$m(y)=g(y)$ for all $y\in P'$ and, if $\xx'$ exists, $m(\xx')=h(\xx)$. Thus it remains to verify that $m$ is
order-preserving on $Q'$. This is clear if $Q'=P'$ since then $m=g$; hence, we may assume that $f(L)\cap f(U)=\es$ and
$Q'=P'\cup\{\xx'\}$. If $y\in P'$ is such that $y<\xx'$, then $y\in f(L)$, so $m(y)=g(y)=gf(\ell)=h(\ell)$ for some
$\ell\in L$. Since $\ell<\xx$ and $h$ is order-preserving on $Q$, we have $m(y)=h(\ell)\leq h(\xx)=m(\xx')$. The
conclusion that $\xx'<z$ for $z\in P'$ implies $m(\xx')\leq m(z)$ follows in a dual fashion. Finally, $m$ preserves the
order on the remaining pairs of distinct elements of $Q'$ as it agrees with $g$ on such pairs; therefore, $m$ is a
poset homomorphism.
\end{proof}

\begin{corollary}[\cite{D}]\label{au}
$\Self(\aleph_0)$ (and so any countable semigroup) embeds into $\End(\mathbb{P})$, the endomorphism monoid of the
generic poset.
\end{corollary}

\begin{proof}
In Corollary \ref{ug} choose $A$ to be the countably infinite anti-chain.
\end{proof}

Actually, in \cite{D} strict endomorphisms of $\mathbb{P}$ were used in order to obtain a more specific version of the
above result, thus avoiding the case when $Q'=P'$ in the above lemma. Still, even the strict endomorphisms were
sufficient to accomplish the desired embedding of $\Self(\aleph_0)$ into $\End(\mathbb{P})$.

\subsection{The case of the (rational) Urysohn space}

The traditional way to work with metric spaces is to consider them as pairs $(X,d_X)$ consisting of a non-empty set of
points $X$ and a distance function (a \emph{metric}) $d_X:X\times X\to\mathbb{R}_0^+$ satisfying the triangle
inequality. However, there is an easy trick to convert them into first-order relational structures by defining, for
each $\alpha\in\mathbb{R}$, $\alpha>0$, a binary relation $R_\alpha$ on $X$ so that $(x,y)\in R_\alpha$ if and only if
$d(x,y)<\alpha$. Furthermore, we may restrict the set of possible distances between points to some countable subset of
$\mathbb{R}_0^+$, thus obtaining structures over a countable signature; for example, the set $\mathbb{Q}_0^+$ of
non-negative rationals will serve this purpose.

In a posthumous paper \cite{Ur} (probably the last contribution before his utterly tragic and untimely death), P.\ S.\
Urysohn showed that up to isometry (a distance preserving bijection) there is a unique complete separable metric space
$\mathbb{U}$ which is ultrahomogeneous and universal, i.e.\ which embeds all separable metric spaces. Today,
$\mathbb{U}$ is called the \emph{Urysohn space} (see, e.g., \cite{HN,Ka,Pe,Us2,Ve} for further background). There is
one construction of $\mathbb{U}$ which will be of a particular interest for us. Namely, let $\C$ be the class of all
finite metric spaces with rational distances; it is easily verified that $\C$ is a \Fr class. Therefore, this class has
a \Fr limit, which is denoted by $\mathbb{U}_\mathbb{Q}$ and called the \emph{rational Urysohn space}. This space is
not complete, but its completion $\ol{\mathbb{U}_\mathbb{Q}}$ has all the defining properties of the Urysohn space, so
$\ol{\mathbb{U}_\mathbb{Q}}$ must be isometric to $\mathbb{U}$ by Urysohn's result.

It was proved in \cite[Lemma 3.5]{D-Berg} that the class of finite metric spaces has the 1PHEP---and the same is true
if we confine ourselves only to spaces with rational distances---implying that both $\mathbb{U}_\mathbb{Q}$ and
$\mathbb{U}$ are homomorphism-homogeneous. Here the notion of a \emph{homomorphism} of metric spaces follows from its
first-order-structure representation: for metric spaces $X,Y$, a function $f:X\to Y$ is a homomorphism if it is
\emph{non-expanding}, that is,
$$d_Y(f(x_1),f(x_2))\leq d_X(x_1,x_2)$$
holds for all $x_1,x_2\in X$. In other words, we have the $1$-Lipschitz mappings.

Here we prove a somewhat stronger result.

\begin{lemma}\label{ms-fun}
Let $\Sigma$ be an additive submonoid of $\mathbb{R}_0^+$. The class $\mathcal{M}_\Sigma$ of all finite metric spaces
with distances in $\Sigma$ enjoys the strict 1PHEP.
\end{lemma}

\begin{proof}
Let $X$ be a finite metric space and $Y = X \cup \{x\st\}$ its one-point extension. Let $f : X \to X'$ be a surjective
non-expanding map. Let $P = X' \cup \{y^*\}$ where $y^* \not\in X'$ is a new point. Let us define a metric on $P$ as
follows: for $x, x' \in X'$ let $d_P(x, x') = d_{X'}(x, x')$, while
$$
d_P(y^*, x) = \min\{d_Y(x^*, w) + d_{X'}(f(w), x) : w \in X\}.
$$
It is easy to verify that $d_P$ is indeed a metric on $P$. Note that since all the distances in $X'$ and $Y$ are from
$\Sigma$, so are all distances in $P$. The mapping $g : Y \to P$ defined by $g(x) = f(x)$ for $x \in X$ and $g(x^*) =
y^*$ is clearly a surjective homomorphism and the diagram
$$
\xymatrix{Y \ar@{->>}[r]^g & P \\
X \ar@{->>}[r]^f \ar@{^{(}->}[u]^{\subseteq} & X' \ar@{^{(}->}[u]_{\subseteq} }
$$
commutes. Let us show that this is a pushout square (in $\mathcal{M}_\Sigma$, but actually it will be a pushout square
in the category of all metric spaces). Let $Z$ be a finite metric space (with distances from $\Sigma$), and let $\mu: Y
\to Z$ and $\nu : X' \to Z$ be homomorphisms such that $\mu|_X = \nu f$:
$$
\xymatrix{& & Z \\
Y \ar@{->>}[r]^g \ar@/^2ex/[urr]^\mu & P \ar[ur]^u & \\
X \ar@{->>}[r]^f \ar@{^{(}->}[u]^{\subseteq} & X' \ar@{^{(}->}[u]_{\subseteq} \ar@/_2ex/[ruu]_\nu & }
$$
Then there exists exactly one mapping $u : P \to Z$ such that the diagram above commutes, namely, the one defined by:
\begin{align*}
u(x) &= \nu(x), \quad (x \in X'),\\
u(y^*) &= \mu(x^*).
\end{align*}
Let us show that $u$ is non-expanding, i.e.\ a homomorphism of metric spaces. Let $x, x' \in X'$. Then $d_Z(u(x),
u(x')) = d_Z(\nu(x), \nu(x')) \leq d_{X'}(x, x') = d_P(x, x')$, since $\nu$ is a homomorphism. To show $d_Z(u(x),
u(y^*)) \leq d_P(x, y^*)$ note that there exists a $w \in X$ such that $d_P(y^*, x) = d_Y(x^*, w) + d_{X'}(f(w), x)$.
Hence,
\begin{align*}
d_P(y^*, x) &= d_Y(x^*, w) + d_{X'}(f(w), x) & \\
            &= d_Z(\mu(x^*), \mu(w)) + d_Z(\nu(f(w)), \nu(x)) & \makebox[75pt][l]{[$\mu$ and $\nu$ are hom's]} \\
            &= d_Z(\mu(x^*), \mu(w)) + d_Z(\mu(w), \nu(x)) & \makebox[75pt][l]{[$\nu f = \mu|_X$]} \\
            &\leq d_Z(\mu(x^*), \nu(x)) & \makebox[75pt][l]{[triangle inequality]} \\
            &= d_Z(u(y^*), u(x)), & \makebox[75pt][l]{[definition of $u$]}
\end{align*}
as required.
\end{proof}

Metric spaces do not admit coproducts---for example, the existence of a hypothetical coproduct of a metric space $X$
and a trivial space would result in a one-point extension of $X$ by a `farthest' possible point, whereas the values
appearing in the vector of distances from the new point to the old ones are obviously not bounded. However, the free
\emph{amalgamated} sum of an amalgam of two spaces exists provided their nonempty intersection is finite.

\begin{lemma}\label{ms-fs}
For any additive submonoid $\Sigma$ of $\mathbb{R}_0^+$, the class $\mathcal{M}_\Sigma$ has the strict AP.
\end{lemma}

\begin{proof}
Let $(X,Y,Z,\mathbf{1}_Y,\mathbf{1}_Y)$ be an amalgam of finite metric spaces with distances from $\Sigma$, so that
$Y=X\cap Z$. We define a metric space $M$ on $X\cup Z$ with distances from $\Sigma$ and prove that $M$ has the
properties required by the definition of the free amalgamated sum $X\ast_YZ$ with respect to $\ol{\mathcal{M}_\Sigma}$.
If either $u,v\in X$, or $u,v\in Z$, then the distance $d_M(u,v)$ coincides with the distance in the corresponding
metric space. If, however, $x\in X$ and $z\in Z$, then we define
$$d_M(x,z)=\min\{d_X(x,y)+d_Z(y,z):\ y\in Y\}.$$
Notice that the indicated minimum exists because $Y$ is finite. A straightforward verification shows that all triangle
inequalities are satisfied, so that we indeed obtain a metric space. Moreover, $M\in\mathcal{M}_\Sigma$.

Assume now that for a (countable) metric space $M'$ (with distances from $\Sigma$) we have given metric space
homomorphisms $f:X\to M'$ and $g:Z\to M'$ such that $f$ and $g$ agree on $Y$. Then the binary relation $h=f\cup g$ is a
function, and we know that it is non-expanding on pairs of points that both belong either to $X$, or to $Z$. Now let
$x\in X\setminus Z$ and $z\in Z\setminus X$. Then $d_M(x,z)=d_X(x,y_0)+d_Z(y_0,z)$ for a particular $y_0\in Y$, so
\begin{align*}
d_{M'}(h(x),h(z))&=d_{M'}(f(x),g(z))\\
&\leq d_{M'}(f(x),f(y_0))+d_{M'}(g(y_0),g(z))\\
&\leq d_X(x,y_0)+d_Z(y_0,z)=d_M(x,z),
\end{align*}
bearing in mind that $f(y_0)=g(y_0)$ and that $f$ and $g$ are homomorphisms (i.e.\ non-expanding functions). Therefore,
$h:M\to M'$ is a homomorphism of metric spaces, and we are done, since $h$ is obviously the unique function extending
both $f$ and $g$ to $M$.
\end{proof}

We can now turn to the main results of this section.

\begin{corollary}\label{rat-ury}
$\Self(\aleph_0)$ (and so any countable semigroup) embeds into $\End(\mathbb{U}_\mathbb{Q})$.
\end{corollary}

\begin{proof}
First of all, note that by Lemmata \ref{fsum} and \ref{ms-fs} if all the distances in $X$ and $Z_i$, $i\in I$, are
rational, so are the distances in $\coprod\st(X,(Y_i,Z_i)_{i\in I})$. Now apply Corollary \ref{ug} while selecting $A$
to be the countably infinite metric space in which the distance between any two different points is $1$ (the unit
$\aleph_0$-simplex) and noticing that any self-map of $A$ is a metric space endomorphism of $A$.
\end{proof}

\begin{proposition}
$\End(\mathbb{U}_\mathbb{Q})$ embeds into $\End(\mathbb{U})$, so the assertion of Corollary \ref{rat-ury} holds for
$\mathbb{U}$ as well.
\end{proposition}

\begin{proof}
Any homomorphism of metric spaces is a uniformly continuous mapping---recall that we are concerned with particular
Lipschitz functions. Therefore, any homomorphism $f:X\to Y$ of metric spaces, where $Y$ is complete, extends uniquely
to the completion $\ol{X}$ of $X$ yielding a (uniformly continuous) function $\ol{f}:\ol{X}\to Y$. Thus if
$x,y\in\ol{X}$ and if $\{x_n\}_{n<\omega}$ and $\{y_n\}_{n<\omega}$ are any Cauchy sequences converging to $x$ and $y$,
respectively, then
\begin{align*}
d(\ol{f}(x),\ol{f}(y)) &= d(\ol{f}(\lim_{n\to\infty}x_n),\ol{f}(\lim_{n\to\infty}y_n))\\
&= d(\lim_{n\to\infty}f(x_n),\lim_{n\to\infty}f(y_n))\\
&= \lim_{n\to\infty}d(f(x_n),f(y_n))\\
&\leq \lim_{n\to\infty}d(x_n,y_n)=d(x,y),
\end{align*}
implying that $\ol{f}$ is a homomorphism.

So, for any $f,g\in\End(\mathbb{U}_\mathbb{Q})$ we have $\ol{f},\ol{g}\in\End(\mathbb{U})$, and both $\ol{gf}$ and
$\ol{g}\ol{f}$ are endomorphisms (thus uniformly continuous functions) of $\mathbb{U}$ extending $gf$. Since
$\mathbb{U}=\ol{\mathbb{U}_\mathbb{Q}}$, we must have $\ol{gf}=\ol{g}\ol{f}$, and the proposition follows.
\end{proof}

\begin{remark}
Since metric space endomorphisms are just a particular type of Lipschitz functions, the monoid of all Lipschitz
functions of both $\mathbb{U}_\mathbb{Q}$ and $\mathbb{U}$ also embeds $\Self(\aleph_0)$.
\end{remark}

\begin{problem}
Is every topological semigroup with a countable base isomorphic to a topological subsemigroup of
$\,\End(\mathbb{U})$\,?
\end{problem}

\subsection{The case of the countable universal ultrahomogeneous semilattice}

In previous applications of our general results we confined ourselves solely to relational structures; here we consider
an example in algebraic ones. However, there is a sense in which semilattices are a unique type of algebra with respect
to the method used in this paper (in particular, in Corollary \ref{ug}), making them similar to relational structures.

\begin{lemma}\label{sl-ug}
Any semilattice $(S,\wedge)$ which is a one-point extension of its subsemilattice $T$ is uniquely generated over $T$.
\end{lemma}

\begin{proof}
Assume $S=\langle T\cup\{x_1\}\rangle=\langle T\cup\{x_2\}\rangle$ for some $x_1,x_2\in S\setminus T$. Bearing in mind
that $\wedge$ is a commutative and idempotent operation, there exist $t,t'\in T$ such that $x_2=x_1\wedge t$ and
$x_1=x_2\wedge t'$. Hence, $x_1\wedge t'=x_1$ and thus
$$x_1 = x_1\wedge t\wedge t' = x_1\wedge t = x_2,$$
as wanted.
\end{proof}

This behavior is actually quite atypical for algebraic structures: in general, a one-point extension of an algebra can
have many relative generators. Basically, the whole of Galois theory is built around the simple idea of subfield-fixing
automorphisms of a given field. Similarly, any $n\times n$ regular diagonal matrix having $1$ as all but one of its
diagonal entries gives rise to a nonidentical automorphism of an $n$-dimensional vector space fixing pointwise one of
its $(n-1)$-dimensional subspaces.

Recall (e.g. from \cite{Gr}) that an algebra $A$ has the \emph{congruence extension property (CEP)} if for any
subalgebra $B$ of $A$ and any congruence $\rho$ of $B$ there exists a congruence $\theta$ of $A$, whose restriction to
$B$ is precisely $\rho$, that is, $\theta\cap(B\times B)=\rho$. Semilattices represent a classical example of algebras
with the CEP, so the following lemma applies in particular to them.

\begin{lemma}\label{sl-cep}
Let $A$ be an algebra with the CEP, and let $B$ be a subalgebra of $A$. Then for every congruence $\rho$ of $B$ the
congruence $\theta$ of $A$ generated by $\rho$ (i.e.\ the smallest $\theta\in\Con A$ containing $\rho$) has the
property that $\theta\cap(B\times B)=\rho$.
\end{lemma}

\begin{proof}
Since $A$ has the CEP, there exists a congruence $\theta'$ of $A$ such that $\theta'\cap(B\times B)=\rho$; but then
$\rho\subseteq\theta'$, implying $\theta\subseteq\theta'$. Therefore,
$$\rho\subseteq\theta\cap(B\times B)\subseteq\theta'\cap(B\times B)=\rho,$$
thus $\theta\cap(B\times B)=\rho$ holds as well.
\end{proof}

The next assertion is a stronger form of Lemma 3.6 from \cite{D-Berg}.

\begin{lemma}
Let $\C$ be a class of finitely generated algebras with the CEP closed under taking homomorphic images. Then $\C$ (and,
in particular, the class of all finite semilattices) has the strict 1PHEP.
\end{lemma}

\begin{proof}
Let $A,B,B'\in\C$ be such that $f:B\to B'$ is a surjective homomorphism and $B$ is a subalgebra of $A$ (the restriction
that $A$ is a one-point extension of $B$ will turn out to be immaterial). Then $\ker f$ is a congruence of $B$. By the
CEP and the previous lemma, if $\theta$ is the congruence of $A$ generated by the relation $\ker f$, then
$\theta\cap(B\times B)=\ker f$. Therefore, $B'\cong B/\ker f$ embeds into $A/\theta$ by the (well-defined) mapping
$\iota:f(b)\mapsto b/\theta$, $b\in B$. On the other hand, the natural homomorphism $\nu_\theta:A\to A/\theta$ is
obviously surjective. Summing up, the following diagram commutes:
$$
\xymatrix{A \ar@{->>}[r]^{\nu_\theta} & A/\theta \\
B \ar@{->>}[r]^f \ar@{^{(}->}[u]^{\subseteq} & B' \ar@{^{(}->}[u]_{\iota} }
$$
By the given conditions, $A/\theta\in\C$. We claim that the above diagram is a pushout square in $\C$.

To this end, assume that $C\in\C$, while $g:B'\to C$ and $f:A\to C$ are two homomorphisms such that $gf=h|_B$:
$$
\xymatrix{& & C \\
A \ar@{->>}[r]^{\nu_\theta} \ar@/^2ex/[urr]^h & A/\theta \ar[ur]^u & \\
B \ar@{->>}[r]^f \ar@{^{(}->}[u]^{\subseteq} & B' \ar@{^{(}->}[u]_{\iota} \ar@/_2ex/[ruu]_g & }
$$
Clearly, a function $u:A/\theta\to C$ makes the above diagram commutative if and only if it satisfies the condition
$u(a/\theta)=h(a)$ for all $a\in A$. It is immediate that there is at most one such function; we argue that such
function exists, i.e.\ that the considered condition consistently defines a function $u$. For this it suffices to show
that $\theta\subseteq\ker h$. However, note that $\ker f\subseteq\ker gf$; thus if $\eta$ is the smallest congruence of
$A$ containing the relation $\ker gf$, then $\theta\subseteq\eta$. But $gf=h|_B$, which implies $\ker gf=\ker
h\cap(B\times B)$ and so $\ker gf\subseteq \ker h$. Since $\ker h$ is a congruence of $A$ it follows that
$\eta\subseteq\ker h$, yielding the desired conclusion $\theta\subseteq\ker h$. Finally, it is routinely verified that
the mapping $u$ is a homomorphism, which completes the proof.
\end{proof}

\begin{remark}
We have already recorded in item 3 of Example \ref{examp} that the \Fr class of finite semilattices has the strict AP.
It is possible, however, to be more specific than this. Namely, it is not difficult to prove the following handy
representation: we can take $\coprod\st(A,(B_i,C_i)_{i\in I})$ to consist of all finite subsets $X$ of
$A\cup\bigcup_{i<\omega}C_i$ such that $|X\cap A|\leq 1$ and $|X\cap C_i|\leq 1$ for all $i<\omega$. Furthermore,
$X\wedge Y$ is defined so that for $C\in\{A\}\cup\{C_i:\ i<\omega\}$ if any of $X\cap C$, $Y\cap C$ is empty, then
$(X\wedge Y)\cap C=(X\cup Y)\cap C$, while otherwise if $X\cap C=\{x\}$ and $Y\cap C=\{y\}$ then $(X\wedge Y)\cap
C=\{x\wedge y\}$.
\end{remark}

By the following main result of this section we supply an affirmative solution to Problem 4.1 from our recent paper
\cite{DM}.

\begin{corollary}
$\Self(\aleph_0)$ (and so any countable semigroup) embeds into $\End(\Omega)$, the endomorphism monoid of the countable
universal ultrahomogeneous semilattice.
\end{corollary}

\begin{proof}
In Corollary \ref{ug} (which applies because of Lemma \ref{sl-ug}), choose the initial semilattice $A$ to be
$F(\aleph_0)$, the free semilattice on a countably infinite set of generators. Namely, any self-map of the free
generating set of $F(\aleph_0)$ induces an endomorphism of $F(\aleph_0)$. Such endomorphisms form a submonoid of
$\End(F(\aleph_0))$ that is isomorphic to $\Self(\aleph_0)$.
\end{proof}

Previously it was only known that $\Self(\aleph_0)$ is a homomorphic image of a subsemigroup of $\End(\Omega)$
\cite[Corollary 4.4]{D-Berg} and that $\End(\Omega)$ embeds all \emph{finite} semigroups \cite[Corollary 4.4]{DM}. We
refer to \cite{AB,DM,DK,DKT} for further properties of $\Omega$, its automorphism group and its endomorphism monoid.

\begin{problem}
Is it possible to embed $\Self(\aleph_0)$ into $\End(\mathbb{D})$ and $\End(\mathbb{A})$, the endomorphism monoids of
the countable universal ultrahomogeneous distributive lattice \cite{DMac} and the countable atomless Boolean algebra,
respectively?
\end{problem}

The main obstacle here is the lack of the property of uniquely generated one-point extensions, i.e the existence of
nontrivial automorphisms of a one-point extension fixing pointwise the initial structure. Of course, other methods
might be developed for approaching this problem; just as a very simple example, let us mention that $\Self(\aleph_0)$
embeds into the endomorphism monoid of the \Fr limit of the class of all finite-dimensional vector spaces over a given
field $\mathbb{F}$. This is because the limit in question is the $\aleph_0$-dimensional vector space over $\mathbb{F}$,
which is at the same time a free algebra in the variety of vector spaces over $\mathbb{F}$, so any self-map of its
basis extends uniquely to a linear map of this space into itself.

\begin{ack}
The research of both authors is supported by Grant No.174019 of the Ministry of Education and Science of the Republic
of Serbia. We are grateful to the anonymous referee whose comments helped to improve the clarity of the presentation.
\end{ack}


\end{document}